\documentclass[reqno,11pt]{amsart}
 
 \usepackage{mathtools}

\newtheorem{theorem}{Theorem}[section]
\newtheorem{lemma}[theorem]{Lemma}
\newtheorem{corollary}[theorem]{Corollary}

\theoremstyle{definition}

\theoremstyle{remark}
\newtheorem{remark}[theorem]{Remark}

 \makeatletter 
 \def\dashint{\operatorname{\,\,\,\mathclap{\int} \kern-.23em\text{\bf--}\!\!}}
 \makeatother

\def\dashnorm{\,\,\text{\bf--}\kern-.5em\|}
\def\ninf{\qopname\relax\@empty{inf\phantom{p}\!\!\!}}

 \makeatother

\newcommand\sfa{{\sf a}}

\newcommand\bC{\mathbb{C}}

\newcommand\bR{\mathbb{R}}
\newcommand\bS{\mathbb{S}}

\newcommand\bZ{\mathbb{Z}}

\newcommand\cF{\mathcal{F}}

\newcommand \cL{\mathcal{L}}

\newcommand\frA{\mathfrak{A}}
\newcommand\frF{\mathfrak{F}}

\newcommand{\loc}{{\rm loc}\,}

 \newcommand{\mysection}[1]{\section{#1}
 \setcounter{equation}{0}}

\begin{document}

\title[On parabolic Adams theorem]
{On the parabolic Adams theorem and
its applications to diffusion processes}
\author{N.V. Krylov}

\email{nkrylov@umn.edu}
\address{School of Mathematics, University of Minnesota, Minneapolis, MN, 55455}
 
 

\begin{abstract} 
We show how the parabolic version of the Adams theorem and its corollary can be used to
estimate in $L_{p}$ the evolution family
associated to a divergence form second-order parabolic 
operator
with parabolic Morrey lower-order terms and also how to 
estimate the moments of the derivatives 
of solutions of It\^o equations
with respect to the initial data when
the drift term has singularities.
\end{abstract}

\maketitle

\mysection{Introduction}

This article is devoted to two issues that
arose in connection with the It\^o stochastic 
equation
\begin{equation}
                        \label{1.30.1}
x_{s}=x+\int_{0}^{s}\sigma(r,x_{r})\,dw_{r}
+\int_{0}^{s}b(r,x_{r})\,dr
\end{equation}
with deterministic $\sigma,b$,
 bounded uniformly nondegenerate $a:=\sigma
\sigma^{*}$ and $b$ from a Morrey class
allowing $b$ satisfying  $|b|\leq\text{const}/|x|$. We deal with $b$ 
which may have singularities in $(t,x)$ sending the readers
to \cite{GG_25} if they are
  interested in cases when  
   $\sigma$ is the unit matrix and
  $b$ is H\"older
continuous in $x$ and is only summable in $t$ to the power $q\in(1,2]$
(as the Hurst exponent is $1/2$)
for a very deep research
of such situations.  The authors of 
\cite{GG_25} also show that one can allow $b$ to be some kind
 of distribution in $x$, but then the equation
is understood in a special way. 

The first issue is the weak uniqueness of
solutions of \eqref{1.30.1}, that is the problem is: are there several solutions
with distinct finite-dimensional distributions?

If we are interested in $\sigma,b$ with low regularity, the best known way to approach
this problem is to use It\^o's formula and
  appropriate existence theorems for differential
equations related to the operator
\begin{equation}
                        \label{1.30.2}
\partial u_{t}+(1/2)a^{ij}D_{ij}u+b^{i}D_{i}u.
\end{equation}

The most general results when $\sigma^{ij}=\delta^{ij}$ and $b$ is in a Morrey class 
(in $(t,x)$) belong to D. Kinzebulatov, see
\cite{Ki_25},  where one finds not only the
necessary PDE results for   operator \eqref{1.30.2} but also existence and weak uniqueness
theorems for equation \eqref{1.30.1}. In case
$\sigma$ has some mild regularity and $b$
is in a Morrey class similar PDE results for
 operator \eqref{1.30.2} are proved in
\cite{Kr_25} without relating them to the problem of weak uniqueness. In the present
article we assume that $a$ is only measurable
but operator \eqref{1.30.2} is given in the divergence form \eqref{1.30.3} and prove what can serve as the key estimate for the PDE results
in \cite{Ki_25} and \cite{Kr_25}.
In the proof of our estimate we use integration
by parts with {\em repeated\/} application
of the parabolic version of the Adams theorem
and its corollary (see Corollary \ref{corollary 10.5,1}). Before, such estimates were obtained
also by integration by parts but with $b$ being
in a Morrey class as a function of $x$
for each $t$.

The second issue related to \eqref{1.30.1}
is the existence of strong solutions, that
is solutions which are functionals of only the trajectory of $w_{t}$. It is shown in
\cite{Kr_25_1} that, somewhat roughly speaking,
 to prove that a given
solution is strong it suffices to include
it into a family of solutions defined for
each $x$ and a family should be such that
its first-order derivatives in $x$ admit
 estimates of its moments.
We show how to obtain such estimates, this
time assuming that $D\sigma$ and $b$ are
in a Morrey class. 
It is done again by integrating by parts
and using the Adams theorem
and its corollary. 
Integration by parts used to be almost
indispensable tool in the past but,
for instance, it was also
used in Section 2.7 of the recent seminal paper
\cite{BFGM_19} and, as our results show, it still remains to be a powerful
tool for obtaining a priori estimates especialy when aided by
the Adans theorem.
Thus, we have a unifying
technique and unifying equation \eqref{1.30.1}.

In the future our results will
allow us to extend the results of 
\cite{Kr_25_2} and \cite{KM_24} on the existence of strong solutions when $D\sigma,b$ for each $t$
are of Morrey class over the case that
$D\sigma,b$  
are of Morrey class with respect to $(t,x)$.

We present our results only for smooth coefficients in order not to obscure
the main ideas with irrelevant details
and leaving it to the interested reader to use approximations and extend them
  to whatever is possible.

In conclusion we introduce some notation.
We let $\bR^{d}$, $d\geq2$, be a Euclidean space
of points $x=(x^{1},...,x^{d})$, $\bR^{d+1}
=\{(t,x):t\in\bR,x\in\bR^{d}\}$. Define $B_{r}=\{x\in\bR^{d}:|x|<r\}$, $B_{r}(x)
=x+B_{r}$,
$$
C_{r}(t,x)=[t,t+r^{2})\times B_{r}(x),
$$
and let $\bC_{r}$ be the collection of $C_{r}(t,x)$. If $C\in\bC_{r}$ by $|C|$ we mean
its Lebesgue measure and for appropriate $f$
and $p\geq1$ we set
$$
\dashint_{C}f\,dxdt=\frac{1}{|C|}
\int_{C}f\,dxdt,\quad \dashnorm f\|^{p}_{L_{p}(C)}=\dashint_{C}|f|^{p}\,dxdt.
$$

For functions $u$ depending on $t,x,\eta$  $(\eta\in\bR^{d})$ set
$$
D_{i}u=u_{x^{i}}=\frac{\partial u}{\partial x^{i}},\quad u_{\eta^{i}}=\frac{\partial u}{\partial \eta^{i}},\quad u_{\eta}=(u_{\eta^{i}}),\quad Du=u_{x}=(D_{i}u),
$$
$$
u_{(\eta)}=\eta^{i}D_{i}u,\quad \partial_{t}u
=\frac{\partial u}{\partial t}.
$$
By $\bS_{\delta}$ we denote the set of symmetric $d\times d$ matrices whose
eigenvalues are in $[\delta,\delta^{-1}]$,
where $\delta\in(0,1]$ is a fixed number.

\mysection{Main results
and discussion} 

We suppose that on $\bR^{d+1}$
we are given an $\bS_{\delta}$-valued  
function $a$ and  $\bR^{d}$-values functions $\sfa,b $. Assume that these functions
belong to $B^{0,\infty}$ which is defined  as the set of 
bounded functions $f(t,x)$ on $\bR^{d+1}$ such that
they are Borel in $t$ and for each $t$ are infinitely differentiable with respect to $x$
with each derivative being a bounded function
on $\bR^{d+1}$.
Set
\begin{equation}
                        \label{1.30.3}
\cL u=\partial_{t}u+(1/2)D_{i}\big(a^{ij}D_{j}u+\sfa^{i}u\big)+b^{i} D_{i}u.
\end{equation}

Fix some  
$$
\rho_{0} \in(0,\infty),\quad p_{0}   \in(2,2+d] .  
$$
Introduce (Morrey norms)
$$
\hat { \sfa }_{p_{0} ,\rho  }=\sup_{r\leq\rho  }r
\sup_{C\in \bC_{r}} 
\dashnorm \sfa \|_{L_{p_{0}}(C)},\quad
\hat b_{p_{0} , \rho  }=\sup_{r\leq\rho  }r
\sup_{C\in \bC_{r}} 
\dashnorm b \|_{L_{p_{0}}(C)}.
$$
 
Our first result consists of the following
in which $f\in C^{\infty}_{0}(\bR^{d})$, $T>0$, and   $u(t,x)$ is the classical solution of
\begin{equation}
                          \label{11.26,5}
\cL u=0\quad\text{in}\quad [0,T]\times \bR^{d}
\end{equation}
with boundary condition $u(T,x)=f(x)$. 
Since all the data are sufficiently regular
the existence of $u$ is a classical result.
\begin{theorem}
                        \label{theorem 11.26,3}
  
Let $n\in\{iI_{i=1}+2iI_{i\geq2},i=1,2,...\}, \lambda\geq 0$.
Then there
are constants $\hat{\sfa},\hat b\in(0,1)$,
depending only on $d,\delta,p_{0}$,  
$n $,
such that, if $\hat { \sfa }_{p_{0},\rho_{0}}\leq 
e^{- \lambda\rho_{0}}\hat{\sfa}$, $\hat b_{p_{0} , \rho_{0} }
\leq e^{- \lambda\rho_{0}}\hat b$, then
\begin{equation}
                           \label{1.31.3}
\int_{\bR^{d}}|u(0,x)|^{2n}e^{-\lambda|x|}\,dx
\leq Ne^{ \alpha T}\int_{\bR^{d}}|f(x)|^{2n}e^{-\lambda|x|}\,dx,
\end{equation}
$$
\int_{[0,T]\times \bR^{d}} u  ^{2n-2}
|Du |^{2}e^{-\lambda|x|}\,dxdt\leq N e^{\lambda
\rho_{0}+\alpha T}
 \int_{\bR^{d}}|f(x)|^{2n}e^{-\lambda|x|}\,dx,
$$
where 
$$
\alpha=N\rho_{0}^{-2}e^{ \lambda\rho_{0}}
$$
and the constants called $N$ depend  only on $d,\delta,p_{0}$,  
$n $.
\end{theorem}

The main emphasis of this result is
on the fact that $N$ is independent of
any smoothness or other characteristics of
$a,\sfa,b$. By using approximations
this allows us to build a solvability theory
(weak solutions),
for instance, 
for equation   with
$|b(x)|\leq c (|x^{1}|+|x^{2}|+|x^{3}|)^{-1}$ which is not
in $L_{3,\loc}$ for any $d\geq3$.

\begin{remark}
                        \label{remark 1.31.2}
In \cite{BFGM_19} the authors were more interested in the estimates like 
\begin{equation}
                           \label{1.31.5}
\int_{\bR^{d}}(1+|x|)^{s}|u(0,x)|^{2n} \,dx
\leq N \int_{\bR^{d}}(1+|x|)^{s}|f(x)|^{2n} \,dx
\end{equation}
for $s\in\bR$ rather than \eqref{1.31.3}.
Actually, \eqref{1.31.5} follows from
\eqref{1.31.3} with $\lambda=1$. Indeed, obviously
$$
\int_{\bR^{d}}(1+|y|)^{s}\int_{\bR^{d}}
|u(0,x)|^{2n}e^{-|x-y|}\,dxdy
$$
$$
\leq N\int_{\bR^{d}}(1+|y|)^{s}\int_{\bR^{d}}
|f(x)|^{2n}e^{-|x-y|}\,dxdy,
$$
which yields \eqref{1.31.5} after using the Fubini theorem and the simple observation that
$$
\int_{\bR^{d}}(1+|y|)^{s}e^{-|x-y|}\,dy
\sim N(1+|x|)^{s}.
$$
\end{remark}
 
To state our second result take an integer $d_{1}\geq d$, 
  a   $d\times d_{1}$-matrix valued 
$\sigma=\sigma(t,x)$ defined on $\bR^{d+1}$
and  an 
 $\bR^{d}$-valued function $b=b(t,x)$   on $\bR^{d+1}$. We assume that $\sigma,b$ are
of class $B^{0,\infty}$ and $a:=\sigma
\sigma^{*}$ is $\bS_{\delta}$-valued. Then let
$w_{t}$ be a $d_{1}$-dimensional Wiener
process on a complete probability
space $(\Omega,\cF,P)$. 

Take $x,\eta\in \bR^{d}$ ,
$t\in\bR$
and consider the following system   
\begin{equation}
                                                        \label{6.20.3}
x_{s}=x+\int_{0}^{s}\sigma (t+r,x_{r})\,dw _{r}+
\int_{0}^{s}b(t+r,x_{r})\,dr,
\end{equation}
\begin{equation}  
                                                        \label{6.20.4}
\eta_{s}=\eta+\int_{0}^{s}\sigma _{(\eta_{r})}(t+r,x_{r})\,dw _{r}
+\int_{0}^{s}b_{(\eta_{r})}(t+r,x_{r})\,dr ,
\end{equation}
where $f_{(\eta)}(t,x)=\eta^{i}D_{i}f(t,x)$.
 
As is well known, \eqref{6.20.3} 
 has a unique solution which we denote by $x_{s}(t,x)$.
By substituting it into \eqref{6.20.4} we see that the coefficients
of \eqref{6.20.4} grow linearly in $\eta$ and hence 
\eqref{6.20.4} also has a unique solution which we denote by
$\eta_{s}(t,x,\eta)$. By the way, observe that equation \eqref{6.20.4}
is linear with respect to $\eta_{r}$. Therefore
$\eta_{t}(x,\eta)$ is an affine function of $\eta$.
For the uniformity of notation we set $x_{s}(t,x,\eta)=x_{s}(t,x)$.
It is also well known 
  (see, for instance, Sections 2.7 and 2.8 of
\cite{Kr_77}) that, as a function of $x$ and $(x,\eta)$, the processes
$x_{s}(t,x)$ and $\eta_{s}(t,x,\eta)$
 are infinitely differentiable in an appropriate sense
(specified below),
their derivatives satisfy the equations which are obtained by formal
differentiation of \eqref{6.20.3} and \eqref{6.20.4},
respectively, and, for any $n\geq0,T\in(0,\infty)$,  $l_{k},\xi_{k}\in\bR^{d}$,
$k=1,...,n$ (if $n\geq1$),
$x,\eta\in \bR^{d}$, $t\in\bR$, and $q\geq 1$,
\begin{equation}
                                                        \label{6.21.1}
 E\sup_{s\leq T}\Big|\Big(\prod_{k=1}^{n}(lb)D _{(l_{k},\xi_{k})}
\Big)(x_{s},\eta_{s})(t,x,\eta)
\Big|^{q}\leq N(1+|\eta|^{m}),
\end{equation}
where $N $ is a certain constant  independent of $(x,\eta)$,
$m=m(n,q)$, and,
for instance,
by $(lb)D _{(l,\xi)} \eta_{s} (t,x,\eta)$ we mean a process $\zeta_{s}$
such that, for any $q\geq1$ and $S\in(0,\infty)$,
$$
\lim_{\varepsilon\downarrow0}
E\sup_{s\leq S}\big|\zeta_{s}-\varepsilon^{-1}
\big(\eta_{s} (t,x+\varepsilon l,\eta+\varepsilon\xi)
-\eta_{s} (t,x,\eta)\big)\big|^{q}=0.
$$
 Introduce
$$
\widehat{D\sigma}_{p_{0},\rho }=\sup_{r\leq\rho  }r
\sup_{C\in \bC_{r}} 
\dashnorm D\sigma \|_{L_{p_{0}}(C)}.
$$
 
\begin{theorem}
               \label{theorem 3.7.1}
Let an integer $\kappa>(d+2)/2$. Then
  there is a constant $N_{0}$, depending only on $d,\delta,p_{0},\kappa$,
  such that if $N_{0}( \widehat{D\sigma}_{p_{0},\rho_{0}}+\hat b_{p_{0},\rho_{0}})
  \leq 1$,
then for each $(t,x)\in\bR^{d+1}$  
the process $x_{s}(t,x)$ admits a modification,
called again $x_{s}(t,x)$, such that
 with probability one $x_{s}(t,\cdot)\in W^{1}_{2\kappa,\loc}(\bR^{d})$ and
\begin{equation}
                       \label{3.7.6}
E \int_{\bR^{d}}e^{-|x|}|Dx_{s}(t,x)|^{2\kappa}\,dx
\leq N ,
\end{equation}
for any $s,|t|\le T\in(0,\infty)$,
where   $N$ depends only on $T$,
$d$, $\delta$, $p_{0}, \rho_{0} $,  $\kappa$.

Furthermore, for each $\alpha<1-(d+2)/(2\kappa)$  and $\omega$
the function $x_{s}(0,x)$
is $\alpha$-H\"older continuous with respect to $x$ and $(\alpha/2)$-H\"older 
continuous with respect to $s$ on each set $[0,T]\times \bar B_{R}$,
$T,R\in(0,\infty)$. 

\end{theorem}  

\begin{remark}
                          \label{remark 1.31.1}
There are not many papers where estimates
of $|Dx_{s}(t,x)|$ are derived in the situation
when the drift term may have singularities.
Still one of them is standing out with a remarkable result on the existence of strong
solutions. This is the paper
\cite{RZ_25} of M. R\"ockner and G. Zhao.
Their estimate stated in Proposition 4.1 of
\cite{RZ_25} is in some aspects stronger than
\eqref{3.7.6} but in some other aspects is weaker.
The essence is nevertheless the same
setting aside the fact that $a^{ij}=\delta^{ij}$ in \cite{RZ_25}. 
However, the assumptions on $b$  of that
proposition are way stronger than ours.
For instance, for $d\geq 3$ the function $b$
such that
$|b|=c(|x^{1}|+|x^{2}|+|x^{3}|)^{-1}$, $c>0$, does not
satisfy the conditions in \cite{RZ_25} and it does satisfy 
the condition in Theorem \ref{theorem 3.7.1} if $c$
is small enough and, say $p_{0}=5/2$.

In \cite{KM_24} the authors made the next
substantial step forward in case $a^{ij}=\delta^{ij}$. They 
follow very closely the probabilistic part of \cite{RZ_25}, 
but make changes in the PDE part of \cite{RZ_25}, to include form-bounded $b$,
 which
is more general than in \cite{RZ_25},
  however this prevents $b$ from having $L_{2}$ singularities in $t$. The latter shortcoming
is promised to be eliminated in a subsequent paper.
\end{remark}

Our third main result is the following
conjecture which allows $b$ to have 
rather strong singularities in $t$
unlike \cite{KM_24}. In this conjecture
we suppose a priori that $\sigma,b$ are only
Borel measurable  and $a=\sigma\sigma^{*}$
is $\bS_{\delta}$-valued. Other
 needed assumptions 
on $\sigma,b$ are included in the statement
of the conjecture. The proof of this conjecture
we have in mind is based on Theorem \ref{theorem 6.21.1}, which allows us to,
basically, reproduce the long scheme
of arguments 
resulting in \cite{Kr_25_2}.
Theorem \ref{theorem 6.21.1}  is also the main tool
in proving Theorem \ref{theorem 3.7.1}.

{\em Conjecture\/}. Suppose that $D\sigma\in L_{1,\loc}(\bR^{d+1})$ and $b=b_{M}+b_{B}$ 
(Morrey part plus bounded part), where both summands are Borel, and introduce
$$
\hat b_{M,p_{0} , \rho  }=\sup_{r\leq\rho  }r
\sup_{C\in \bC_{r}} 
\dashnorm b_{M} \|_{L_{p_{0}}(C)},
\quad
\bar b_{B}(t)=\sup_{x\in\bR^{d}}|b_{B}(t,x)|.
$$

 The conjecture is that there is a constant $N_{0}=N_{0}(d, \delta,
p_{0})$   such that if
$$
N_{0}(\widehat{D\sigma}_{p_{0},\rho_{0}}+\hat b_{M,p_{0},\rho_{0}})\leq1 ,\quad \int_{\bR}\bar b^{2}_{B}(t)\,dt<\infty,
$$
then equation \eqref{6.20.3} has a solution,
which is $\{\cF^{w}_{t}\}$-adapted, where
$\cF^{w}_{t}$ is the completion of
$\sigma\{w_{s},s\leq t\}$.

Our results are based on some properties
of the parabolic Riesz potentials
and a deep theorem of Adams on their
estimates.

\mysection{Parabolic Riesz potentials}

For $k,s,r,\alpha>0  $, and appropriate $f(t,x)$'s
on $\bR^{d+1}$
\index{$S$@Miscelenea!$p_{\alpha,k}(s,r)$}%
\index{$C$@Operators!$P_{\alpha,k}f(t,x)$}%
 define
$$
p_{\alpha,k}(s,r)=\frac{1}{s^{(d+2-\alpha)/2}}e^{-r^{2}/(ks)}I_{s>0}, 
$$
$$
P_{\alpha,k}f(t,x)=\int_{\bR^{d+1} }p_{\alpha,k}(s,|y|)f(t+s,x+y)\,dyds.
$$
$$
=\int_{t}^{\infty}\int_{\bR^{d} }p_{\alpha,k}(s-t,|y-x|)
f(s,y)\,dsdy.
$$  

\begin{theorem}
                     \label{theorem 10.5,1}
(i) There is a constant $c(d)>0$ such that
$u=c(d)P_{2,4}(\partial_{t}u+\Delta u)$
if $u\in C^{\infty}_{0}(\bR^{d+1})$.

(ii) For $\alpha,\beta,k>0$ we have
$P_{\alpha,k}P_{\beta,k}=c(\alpha,\beta,k)P_{\alpha+\beta,k}$. 

(iii) For any integer $n\geq1$, $\alpha>n$, and bounded $f$
with compact support we have $|D^{n}P_{\alpha,k}f|\leq N(d,\alpha,n)P_{\alpha-n,2\kappa}|f|$.

\end{theorem}

Proof. Assertion (i) follows from It\^o's
formula applied to $u(t,\sqrt 2 w_{t})$, where
$w_{t}$ is a $d$-dimensional Wiener process.
Assertion (ii) follows after direct computations. Assertion (iii) is also
proved by direct computations augmented
by the fact that $r^{m}e^{-r^2/\kappa}\leq
N(m,\kappa)e^{-r^2/(2\kappa)}$. \qed

To state a parabolic analog of one of  the Adams theorems, we need to introduce
the homogeneous version of
$$
\hat f_{p_{0},\rho }:=\sup_{r\leq\rho  }r
\sup_{C\in \bC_{r}} 
\dashnorm f \|_{L_{p_{0}}(C)},
$$
namely,
\begin{equation}
                             \label{2.16.5}
\|f\|_{\dot E_{p, \beta} }:=
\sup_{\rho >0,C\in\bC_{\rho}}\rho^{\beta}
\dashnorm f   \|_{ L_{p }(C)} <\infty ,
\end{equation}
where $p>1,\beta>0$. We write $f\in \dot E_{p, \beta}$ if \eqref{2.16.5} holds.
Observe that if $f\in \dot E_{p, \beta}$
and $p\beta >d+2$, then $f=0$ (a.e.).

\begin{remark}
                         \label{remark 11.30,1}
If $C\in\bC_{\rho_{0}}$, then
as is not hard to prove
$$
\|I_{C}b\|_{\dot E_{p_{0}, 1}}\leq
\hat b_{p_{0} ,\rho_{0} },\quad \|I_{C}\|
_{\dot E_{p_{0}, 1}}\leq \rho_{0}.
$$
 
\end{remark}

Here is a parabolic analog 
of one of theorems of Adams.                  
 
\begin{theorem}[Theorem 4.5 \cite{Kr_25}]
                        \label{theorem 5.25,1}
Let $ \alpha>0,1<q<p<\infty$, $k>0$, $b(t,x)\geq0$. Then for any $f(t,x)\geq0$
\begin{equation}
                            \label{5.25,1}
\|bP_{\alpha,k}f\|_{L_{q }}
\leq N\|b\|_{\dot E_{p, \alpha} }
\|f\|_{L_{q }},
\end{equation}
where $N$ depends only on $d,q ,p,\alpha,k$.
 
\end{theorem}

\begin{corollary}
                 \label{corollary 10.5,1}
Estimate \eqref{5.25,1} says that the operator
$f\to bP_{\alpha,k}f$ is bounded in $L_{q}$. Its conjugate (with time reversed)
is then also bounded as an operator in 
$L_{q' }$, where $q '=q 
/(q -1)$, that is
$$
\|P_{\alpha,k}(bf)\|_{L_{q' }}
\leq N\|b\|_{\dot E_{p,\alpha} }
\|f\|_{L_{q' }}.
$$
In case $q =2$ and $\alpha=1$ we have that, if $p>2$, then
$$
\|P_{1,k}(bf)\|_{L_{2} }
\leq N\|b\|_{\dot E_{p,1} }
\|f\|_{L_{2} }. 
$$
\end{corollary}

A useful addition to the above properties
of multiplication by $b$ is the following.

\begin{lemma}
                          \label{lemma 11.29,1}
Let $p>2$, $f(x)\geq0$, $b\geq 0$, $T\in \bR$, $c=(4\pi)^{-d/2}$
,$$
u(t,x)=\frac{c}{(T-t)^{d/2}}\int_{\bR^{d}}e^{-|x-y|^{2}/(4T-4t )}f(y)\,dy\, I_{t<T}+I_{t=T}f(x).
$$
Then
\begin{equation}
                                \label{11.29,2}
\int_{ \bR^{d+1}}
b^{2}u^{2}\,dxdt\leq N(d,p)\|b\|^{2}_{\dot E_{p,1} }\int_{\bR^{d}}f^{2}\,dx.
\end{equation}
\end{lemma}

Proof. We may assume that $T=0$ and $f$ is smooth
and bounded. In that case
set $v(t,x)=u(-t,x)$, $w(t,x)=v(t,x)\zeta(t)$, $t\geq0$,
where $\zeta$ is infinitely differentiable, $\zeta=1$ near $(-\infty,0]$, $\zeta(t)=0$ for $t\geq 1$,
$\zeta\geq0,\zeta'\leq0$.
Observe that $\partial_{t}u+\Delta u=0$ for $t<0$,
$\partial_{t}v=\Delta v$ for $t>0$,
$\partial_{t}w=\Delta w+v\zeta'$ for $t>0$,
which after being multiplied by $w$ and integrating by parts 
yields
\begin{equation}
                             \label{11.29,3}
\int_{\bR^{d+1}_{0}}|Dw|^{2}\,dxdt=
(1/2)\int_{\bR^{d}}f^{2}\,dx
+\int_{0}^{\infty}\Big(\int_{\bR^{d}_{0}}v^{2}\,dx
\Big)
\zeta\zeta'\,dt\leq (1/2)\int_{\bR^{d}}f^{2}\,dx,
\end{equation}
where $\bR^{d+1}_{t}:=(t,\infty)\times\bR^{d}$.

By It\^o's formula for $t\geq0$ we have
$$
w(t,x)=-cP_{2,4}(\partial_{t}w+\Delta w)(t,x)=-
cP_{2,4}(2\Delta w+v\zeta')(t,x).
$$
For us the most important is that this holds with
$t=0$ when $w(0,x)=f(x)$. Now
  the semigroup property of the heat semigroup 
  implies that
  for $t<0$ we have
$$
u(t,x)=
\frac{c}{(-t)^{d/2}}\int_{\bR^{d}}
e^{-|x-y|^{2}/(-4t)}w(0,y)\,dy
=-cP_{2,4}\big((2\Delta w+v\zeta')I_{0,\infty)}\big)(t,x).
$$
 
Next, we use that $|P_{2,4}\Delta w|=|(D_{i}
P_{2,4}D_{i} w|\leq NP_{1,8}|Dw|$ and
$P_{2,4}=NP_{1,8}P_{1,8}$ combined with the fact 
that, obviously, $\zeta'\in \dot E_{p,1}$.
Then by applying Theorem \ref{theorem 5.25,1} and Corollary
\ref{corollary 10.5,1} we arrive at
$$
\int_{(-\infty,0)\times\bR^{d}}
b^{2}\big(P_{2,4}(I_{0,\infty)}\Delta w)\big)^{2}\,dxdt
\leq N\| bP_{1,8}| I_{0,\infty)}Dw|\,\|^{2}_{L_{2} }
$$
$$
\leq N\|b\|^{2}_{\dot E_{p,1} }
\|Dw\|^{2}_{L_{2}(\bR_{0}^{d+1})},
$$
$$
\int_{(-\infty,0)\times\bR^{d}}
b^{2}\big(P_{2,4}|I_{0,\infty)}v\zeta'|\big)^{2}\,dxdt
\leq N\|b\|^{2}_{\dot E_{p,1} }
\|P_{1,8}|I_{0,\infty)}v\zeta'|\,\|^{2}_{L_{2} }
$$
$$
\leq N\|b\|^{2}_{\dot E_{p,1} }\|vI_{(0,1)}\|
^{2}_{L_{2} }.
$$
After that it only remains to use \eqref{11.29,3}
and   that, for any $t>0$,
$$
\int_{\bR^{d}}v^{2}(t,x)\,dx\leq \int_{\bR^{d}}f^{2}\,dx.
$$
The lemma is proved. \qed

\mysection{Proof of Theorem \ref{theorem 11.26,3}} Without restricting   generality
we assume that $\hat {\sfa }_{p_{0},\rho_{0}}
\leq 1$, $\hat b_{p_{0} , \rho_{0} }
\leq 1$. 

{\em Step 1. Integrating by parts\/}. Take a $C\in\bC_{\rho_{0}}$ and 
a nonnegative $\zeta
\in C^{\infty}_{0}(C)$  with the integral
of its square equal to one.
We    multiply  \eqref{11.26,5}
by $\zeta^{2} u^{2n-1}$ and integrate by parts. Then noting that $a^{ij}D_{i}uD_{j}u\geq\delta|Du|^{2}$ and for any $\varepsilon>0$
$$
\int_{[s,T]\times \bR^{d}} |u^{n}\zeta^{-1}D(\zeta^{2} )  |\zeta|u^{n-1} Du| \,dxdt
\leq \varepsilon\int_{[s,T]\times \bR^{d}}\zeta^{2} u  ^{2n-2}
|Du |^{2}\,dxdt
$$
$$
+\varepsilon^{-1}\int_{[s,T]\times \bR^{d}}
|D\zeta|^{2}u^{2n}\,dxdt,
$$
we find for $s\leq T$ that
$$
\int_{\bR^{d}}\zeta^{2}(s,x)u^{2n}(s,x)\,dx+(\delta/4)
\int_{[s,T]\times \bR^{d}}\zeta^{2} u  ^{2n-2}
|Du |^{2}\,dxdt
$$
$$
\leq \int_{\bR^{d}}\zeta^{2}(T,\cdot)f ^{2n }   \,dx +N\int_{[s,T]\times \bR^{d}}
|D\zeta|^{2}u^{2n}\,dxdt+I_{0}+
I_{1}+I_{2}+I_{3},
$$
where
$$
I_{0}=-\int_{[s,T]\times \bR^{d}}
u^{2n}\partial_{t}\zeta^{2}\,dxdt,
$$
$$
2 I_{1}=- \int_{[s,T]\times \bR^{d}}\zeta \sfa^{i}u^{n}D_{i} \zeta u^{n}\,dxdt
$$
$$
\leq
\int_{[s,T]\times \bR^{d}}\big(|D \zeta|^{2}  u^{2n}+  \zeta ^{2}|\sfa|^{2}u^{2n}\big)\,dxdt,
$$
$$
2I_{2}=-(2n-1)\int_{[s,T]\times \bR^{d}}\zeta^{2}u^{2n-1}\sfa^{i}D_{i}u\,dxdt
$$
$$
\leq (\delta/8)\int_{[s,T]\times \bR^{d}}\zeta^{2} u  ^{2n-2}
|Du |^{2}\,dxdt+N\int_{[s,T]\times \bR^{d}} \zeta^{2}|\sfa|^{2}u^{2n} \,dxdt,
$$
$$
I_{3}=\int_{[s,T]\times \bR^{d}}\zeta^{2} u  ^{2n-1}b^{i}D_{i}u\,dxdt
$$
$$
\leq (\delta/16)\int_{[s,T]\times \bR^{d}}\zeta^{2} u  ^{2n-2}
|Du |^{2}\,dxdt+N\int_{[s,T]\times \bR^{d}} \zeta^{2}|b|^{2}u^{2n} \,dxdt,
$$
 
It follows that
$$
\int_{\bR^{d}}\zeta^{2}(s,x)u^{2n}(s,x)\,dx+(\delta/16)
\int_{[s,T]\times \bR^{d}} \zeta^{2} u  ^{2n-2}
|Du |^{2}\,dxdt
$$
$$
\leq  \int_{\bR^{d}}\zeta^{2}(T,\cdot)f ^{2n }   \,dx
+N \int_{[s,T]\times \bR^{d}}\big(|D\zeta |^{2}
 +\rho^{-2}\zeta^{2}+|\partial_{t}\zeta^{2}|\big) u^{2n}\,dxdt
$$
\begin{equation}
                               \label{11.28,4}
+N\int_{[s,T]\times \bR^{d}} ( |\sfa|+ |b|)^{2}\zeta^{2}u^{2n} \,dxdt.
\end{equation}

Before proceeding further we note that we may
look at $\zeta$ as a scaled and translated
function with support in $C_{1}$. Then it is seen that 
\begin{equation}
                              \label{1.21.3}
\rho_{0}^{d+2} |\zeta |^{2}
+
\rho_{0}^{d+4} |D\zeta |^{2}
+
\rho_{0}^{d+6} |\partial_{t}\zeta |^{2}
 \leq N (d),
\end{equation}
and we infer from \eqref{11.28,4} that
$$
\int_{\bR^{d}}\zeta^{2}(s,x)u^{2n}(s,x)\,dx+(\delta/16)
\int_{[s,T]\times \bR^{d}} \zeta^{2} u  ^{2n-2}
|Du |^{2}\,dxdt
$$
$$
\leq  \int_{\bR^{d}}\zeta^{2}(T,\cdot)f ^{2n }   \,dx
+N \rho^{-d-4}\int_{[s,T]\times \bR^{d}} I_{C} u^{2n}\,dxdt
$$
\begin{equation}
                               \label{1.13.2}
+N \int_{[s,T]\times \bR^{d}}   ( |\sfa|+ |b|)^{2}\zeta^{2}u^{2n} \,dxdt.
\end{equation}
 
{\em Step 2. Using Adams's theorem\/}.
To estimate the last term it is convenient
to transform \eqref{11.26,5} and define $u(t,x)=0$ for $t>T$. 
For $v:=u^{n}$ we have
$$
\partial_{t}(\zeta v)+ \Delta(\zeta v)+2\zeta\Big(\frac{1}{n}-1\Big)a^{ij}(D_{i}(u^{n/2}))D_{j}(u^{n/2})
$$
$$
+(1/2)\zeta D_{i}(a^{ij}D_{j}v)  
+\zeta  (b^{i}+(1/2)\sfa^{i}) D_{i}v -v\partial_{t}\zeta- \Delta(\zeta v)+(n/2)\zeta 
(D_{i}\sfa^{i})v=0 
$$
for $t\leq T$ with boundary value $\zeta v|_{t=T}
=\zeta(T,\cdot)f^{n}$.

Here the third term is either zero if $n=1$
or negative if $n\geq 2$ when $v\geq0$
($n$ is even).
Then by It\^o's formula, applied to $(\zeta v)
(t+r,x+\sqrt{2}w_{t})$, where $w_{t}$ is 
a $d$-dimensional Wiener process,
we get  that
$$
\int_{[s,T]\times \bR^{d}}|\sfa|^{2}\zeta^{2}v^{2}\,dxdt
\leq N \int_{\bR^{d+1}_{s}}|\sfa|^{2}I_{C  }P^{2}_{2,4}(F)\,dxdt
$$
\begin{equation}
                                 \label{11.28,5}
+
N \int_{[s,T]\times \bR^{d}}|\sfa|^{2}I_{C }
\hat T^{2}_{T-t}[\zeta(T,\cdot)f^{n}](x)\,dxdt,
\end{equation}
where $\hat T_{r}h(x)=Eh(x+\sqrt2 w_{r})$,
$$
F=(1/2)\zeta D_{i}(a^{ij}D_{j}v)  
+\zeta  (b^{i}+(1/2)\sfa^{i}) D_{i}v -v\partial_{t}\zeta- \Delta(\zeta v)+(n/2)\zeta 
(D_{i}\sfa^{i})v.
$$

Having in mind that $v(t,x)=0$ for $t>T$, note that 
$$
P_{2,4}(\zeta D_{i}(a^{ij}D_{j}v))=
 D_{i}\big(P_{2,4}(\zeta a^{ij} D_{j}v)\big)
-NP_{1,8}P_{1,8}\big(I_{C}(D_{i} \zeta) a^{ij}   D_{j}v\big)
$$
and since $|DP_{2,4}h|\leq NP_{1,8}|h|$, we have by Theorem
\ref{theorem 5.25,1}  
$$
\int_{\bR^{d+1}_{s}}|\sfa|^{2}I_{C  }
\big| D_{i}P_{2,4}(\zeta a^{ij} D_{j}v )\big  |^{2}\,dxdt
\leq N\int_{\bR^{d+1} }|\sfa|^{2}I_{C  }
  P_{1,8}^{2}(\zeta |a^{ij} D_{j}v| I_{(s,\infty)}) \,dxdt
$$
$$
\leq N\hat\sfa_{p_{0},\rho_{0}}^{2}
\int_{[s,T]\times \bR^{d}}\zeta^{2}|Dv|^{2}\,dxdt\leq N\hat \sfa_{p_{0},\rho_{0}}^{2}\rho_{0}^{-d-2}
\int_{[s,T]\times \bR^{d}}I_{C}|Dv|^{2}\,dxdt.
$$
Similarly, invoking also Corollary \ref{corollary 10.5,1} and Remark \ref{remark 11.30,1} we get 
$$
\int_{\bR^{d+1}_{s}}|\sfa|^{2}I_{C  }
\Big(P_{1,8}P_{1,8}\big(I_{C}|D \zeta  |\, Dv| \big)\Big)^{2}\,dxdt
$$
$$
\leq N\hat \sfa_{p_{0},\rho_{0}}^{2}
\int_{\bR^{d+1}_{s}}\Big( P_{1,8}\big(I_{C}|D \zeta  |\, Dv| \big)\Big)^{2}\,dxdt 
$$
$$
\leq
N\hat \sfa_{p_{0},\rho_{0}}^{2}\rho_{0}^{-d-2}\int_{[s,T]\times \bR^{d}}I_{C}|Dv|^{2}\,dxdt.
$$
In the same way applying Theorem
\ref{theorem 5.25,1} and Corollary \ref{corollary 10.5,1} we obtain
$$
\int_{\bR^{d+1}_{s}}|\sfa|^{2}I_{C }
 P^{2}_{2,4} (\zeta b^{i}D_{i}v  ) \,dxdt\leq N\hat \sfa_{p_{0},\rho_{0}}^{2}
\rho_{0}^{-d-2}\int_{[s,T]\times \bR^{d}}I_{C}|Dv|^{2}\,dxdt,
$$ 
$$
\int_{\bR^{d+1}_{s}}|\sfa|^{2}I_{C }
 P^{2}_{2,4} (\zeta \sfa^{i}D_{i}v ) \,dxdt\leq N\hat \sfa_{p_{0},\rho_{0}}^{2}
\rho_{0}^{-d-2}\int_{[s,T]\times \bR^{d}}I_{C}|Dv|^{2}\,dxdt.
$$ 

Using again the fact that $I_{C}\in\dot E_{p_{0} ,1}$ yields (we also use that
$\hat \sfa_{p_{0},\rho_{0}} \leq 1$)
$$
\int_{\bR^{d+1}_{s}}|\sfa|^{2}I_{C }
 P^{2}_{2,4} (I_{C}v \partial_{t}\zeta) \,dxdt\leq N \int_{[s,T]\times \bR^{d}}\rho_{0}^{2}|\partial_{t}\zeta|^{2}| v|^{2}\,dxdt
$$
$$
\leq N \rho_{0}^{-d-4}\int_{[s,T]\times \bR^{d}}I_{C} v^{2}\,dxdt.
$$
The following is now routine
$$
\int_{\bR^{d+1}_{s}}|\sfa|^{2}I_{C }
 P^{2}_{2,4} (\Delta(\zeta v) ) \,dxdt
$$
$$
\leq 
\int_{\bR^{d+1}_{s}}|\sfa|^{2}I_{C }
P^{2}_{1,8}(v|D\zeta| +\zeta|Dv| )\,dxdt
$$
$$
\leq
N\hat \sfa_{p_{0},\rho_{0}}^{2}\int_{[s,T]\times \bR^{d}} ( |D\zeta|^{2} v ^{2}+\zeta^{2}|Dv|^{2})\,dxdt
$$
$$
\leq
N\hat \sfa_{p_{0},\rho_{0}}^{2}\int_{[s,T]\times \bR^{d}} ( \rho_{0}^{-d-4} v ^{2}+\rho_{0}^{-d-2}|Dv|^{2})\,dxdt.
$$
 
To deal with the last term entering $F$, observe
that 
$$
P _{2,4} \big(\zeta (D_{i}\sfa^{i})v \big)=D_{i}P _{2,4} 
\big(\zeta  \sfa^{i} v   \big)
$$
$$
-P _{2,4} \big((D_{i}\zeta)  \sfa^{i} v \big)-
P _{2,4} \big(\zeta  \sfa^{i} D_{i}v \big),
$$
 where the last term  has been already dealt with above. Furthermore,
$$
\int_{\bR^{d+1}_{s}}|\sfa|^{2}I_{C }
 \big(D_{i}P _{2,4} \big(\zeta  \sfa^{i} v  \big)\big)^{2}\,dxdt 
$$
$$
\leq N\int_{\bR^{d+1}_{s}}|\sfa|^{2}I_{C }
  P ^{2}_{1,8} \big(\zeta  \sfa^{i} v   \big) \,dxdt\leq
N\hat \sfa_{p_{0},\rho_{0}}^{2}\int_{[s,T]\times \bR^{d}}
|\sfa|^{2}\zeta^{2}v^{2}\,dxdt,
$$
$$
\int_{\bR^{d+1}_{s}}|\sfa|^{2}I_{C }
P^{2} _{2,4} \big((D_{i}\zeta)  \sfa^{i} v \big)\,dxdt
$$
$$
\leq
N \int_{\bR^{d+1}_{s}}P^{2} _{1,8}
 \big(  |\sfa||vD\zeta|  \big)\,dxdt
$$
$$
\leq
N \int_{[s,T]\times \bR^{d}}   |D\zeta|^{2}v^{2}  \,dxdt
\leq
N \rho_{0}^{-d-4}\int_{[s,T]\times \bR^{d}}  v^{2}  \,dxdt.
 $$

To finish dealing with \eqref{11.28,5} we apply
Lemma \ref{lemma 11.29,1}
to estimate the last term and get 
$$
\int_{[s,T]\times \bR^{d}}|\sfa|^{2}\zeta^{2}v^{2}\,dxdt
\leq N \int_{\bR^{d}}\zeta^{2}(T,\cdot)
f^{2n}\,dx
$$
$$
 + N\hat \sfa_{p_{0},\rho_{0}}^{2}\rho_{0}^{-d-2}
\int_{[s,T]\times \bR^{d}}I_{C}|Dv|^{2}\,dxdt
+N \rho_{0}^{-d-4}\int_{[s,T]\times \bR^{d}}I_{C} v^{2}\,dxdt
$$
$$
+N_{1}\hat \sfa_{p_{0},\rho_{0}}^{2}\int_{[s,T]\times \bR^{d}}|\sfa|^{2}\zeta^{2}v^{2}\,dxdt.
$$
Requiring
\begin{equation}
                           \label{1.14,1}
N_{1}\hat \sfa_{p_{0},\rho_{0}}^{2}\leq 1/2,
\end{equation}
allows us to eliminate the last term in the preceding estimate. After that observing that
estimating the integral of $| b|^{2}\zeta^{2}v^{2}
$ is not much different and coming back to
\eqref{1.13.2}
yields
$$
\int_{\bR^{d}}\zeta^{2}(s,x)u^{2n}(s,x)\,dx+
\int_{[s,T]\times \bR^{d}} \zeta^{2} u  ^{2n-2}
|Du |^{2}\,dxdt
$$
$$
\leq N\int_{\bR^{d}}\zeta^{2}(T,\cdot)f ^{2n }   \,dx
+N \rho_{0}^{-d-4}\int_{[s,T]\times \bR^{d}} I_{C} u^{2n}\,dxdt
$$
\begin{equation}
                               \label{1.14.2}
 +N\big(\hat \sfa_{p_{0},\rho_{0}}^{2}
+\hat b_{p_{0},\rho_{0}}^{2}\big)\rho_{0}^{-d-2}
\int_{[s,T]\times \bR^{d}} I_{C} u  ^{2n-2}
|Du |^{2}\,dxdt.
\end{equation}

{\em Step 3. Using the arbitrariness of $C$\/}.
We substitute here $C=C_{\rho_{0}}(\tau,\xi)$
and $\zeta(t-\tau,x-\xi)$ in place of $C$ and $\zeta(t,x)$,
where $(\tau,\xi)\in\bR^{d+1}$. Then we multiply
both parts by $e^{-\lambda |\xi|}$ and integrate
through the resulting inequality with respect to
$(\tau,\xi)\in\bR^{d+1}$. 
At this point it is worth mentioning that since $f\in C^{\infty}_{0}(\bR^{d})$ and $a,\sfa$ and $b$
are sufficiently regular, $u$ and its derivatives
go to zero as $|x|\to\infty$ exponentially fast.
Therefore, our manipulations are well justified.
 
Note that
$$
e^{\lambda \rho_{0}}e^{-\lambda |x|}\geq \int_{\bR^{d+1}}\zeta^{2}(t-\tau,x-\xi)e^{-\lambda |\xi|}\,d\xi d\tau\geq  
e^{-\lambda \rho_{0}}e^{-\lambda |x|},
$$
 $$
e^{\lambda \rho_{0}}e^{-\lambda |x|}\geq N(d)\int_{\bR^{d+1}}\rho_{0}^{-d-2}I_{C_{\rho_{0}}}(t-\tau,x-\xi)e^{-\lambda |\xi|}\,d\xi d\tau\geq  
e^{-\lambda \rho_{0}}e^{-\lambda |x|}.
$$

Consequently,
$$
e^{\lambda\rho_{0}}\int_{\bR^{d}}u^{2n}(s,x)e^{-\lambda|x|}\,dx
+\int_{[s,T]\times \bR^{d}}  u  ^{2n-2}
|Du |^{2}e^{-\lambda|x|}\,dxdt
$$
$$
\leq Ne^{ \lambda \rho_{0}} \int_{\bR^{d}}f^{2n}
e^{-\lambda|x|}\,dx +N e^{2\lambda \rho_{0}}
\rho_{0}^{-2}\int_{[s,T]\times \bR^{d}}  u^{2n}e^{-\lambda|x|}\,dxdt
$$
\begin{equation}
                                 \label{12.9,4}
+
N_{2}e^{2\lambda \rho_{0}}\big(\hat \sfa_{p_{0},\rho_{0}}^{2}
+\hat b_{p_{0},\rho_{0}}^{2}\big)
\int_{[s,T]\times \bR^{d}} 
   u  ^{2n-2}
|Du |^{2}e^{-\lambda|x|}\,dxdt.
\end{equation}
The last term is absorbed into the left-hand side
if we require
$$
N_{2}e^{2\lambda \rho_{0}}\big(\hat \sfa_{p_{0},\rho_{0}}^{2}
+\hat b_{p_{0},\rho_{0}}^{2}\big)\leq 1/2
$$
and to deal with the first term on the left   we use Gronwall's inequality after
 throwing away the second term on the left. The theorem is proved. \qed

\begin{remark}
                        \label{remark 1.31.5}
It is probably worth mentioning one case
of estimates like \eqref{1.31.3} even though it is
unrelated to Morrey spaces. We mean the case
when
$$
\int_{\bR}\sup_{\bR^{d}}(|\sfa|+|b|)^{2}(t,x)\,
dt<\infty.
$$
In that case you just stop after \eqref{11.28,4} and use Gronwall's inequality.
\end{remark}

\mysection{Regularity of solutions of
stochastic equations as functions of initial data}

Here we start on our way to prove Theorem
\ref{theorem 3.7.1} and we work in the setting
it is stated in.

\begin{lemma} 
                                                     \label{lemma 6.21.1}
 Let $\eta\in\bR^{d}$   and
 $\xi_{s}(t,x,\eta)=(lb)D_{\eta}x_{s}(t,x)$. Then

(i) $\xi_{s}(t,x,\eta)$ satisfies \eqref{6.20.4}, hence, 
coincides with $\eta_{s}(t,x,\eta)$ for every $(t,x,\eta)$
with probability one for all $s$.

(ii) If $f(x)$ is infinitely differentiable with bounded derivatives,
then
\begin{equation}
                                                        \label{6.21.5}
Ef_{(\xi_{s}(t,x,\eta))}(x_{s}(t,x))\Big(=
E\big(f_{(\xi_{s}(t,x,\eta))}\big)(x_{s}(t,x))\Big)
=\big(Ef(x_{s}(t,x))\big)_{(\eta)}.
\end{equation}
\end{lemma}

Proof. Assertion (i) is alluded to before
Theorem  \ref{theorem 3.7.1} and is well known (see, for instance, \cite{Kr_77}).
  Assertion (ii)   follows from (i)
and the fact that (see, for instance, \cite{Kr_77})
$$
\big(Ef(x_{s}(t,x))\big)_{(\eta)}=Ef_{(\xi_{s}(t,x,\eta))}(x_{s}(t,x)).
$$
The lemma is proved.\qed

Thus $\eta_{s}$ is the first derivative of $x_{s}(t,x)$ in $x$ and
our goal in this section is to estimate
$E|\eta_{s}(t,x,\eta)|^{2n}$ without using
the information on the smoothness of $\sigma,b$
but rather based only on the Morrey norms of $D\sigma,b$ (and not $Db$).

Here is the main starting point.
\begin{lemma}
                                                     \label{lemma 6.21.10}
Let $f(x,\eta)$ be infinitely differentiable and such that
each of its derivatives grows  as $|x|+|\eta|\to\infty$
not faster than polynomially. Let $T\in\bR$. Then

(i) for $t\leq T$, the function
$u(t,x,\eta):=Ef\big((x_{T-t},\eta_{T-t})(t,x,\eta)\big)$
is infinitely differentiable in $(x,\eta)$ and 
each of its derivatives   by absolute
value is bounded on each finite   interval in
$(-\infty,T]$
by a constant times $(1+|x|+|\eta|)^{m}$
for some $m$,

(ii) for each $x,\eta$ the function $u(t,x,\eta)$ is Lipschitz continuous with respect to $t\in[0,T]$, 

(iii) in $(0,T)\times \bR^{2d}$  (a.e.)
$\partial_{t}u(t,x,\eta)$ exists and
$$
0= \partial_{t}u(t,x,\eta)+ (1/2)\sigma^{ik}\sigma^{jk}(t,x)u_{x^{i}x^{j}} (t,x,\eta)
+\sigma^{ik}\sigma_{(\eta)}^{jk}(t,x)u_{x^{i}\eta^{j}} (t,x,\eta)
$$
$$
+(1/2)\sigma_{(\eta)}^{ik} \sigma_{(\eta)}^{jk}(t,x)u_{\eta^{i}\eta^{j}}(t,x,\eta)
+b^{i}(t,x)u_{x^{i}} (t,x,\eta)+b^{i}_{(\eta)}(t,x)u_{\eta^{i}} (t,x,\eta)
$$
\begin{equation}
                          \label{6.21.3}
=:\partial_{t}u(t,x,\eta)+\check \cL(t,x,\eta)u(t,x,\eta).
\end{equation}
\end{lemma}

Proof. Assertion (i) is known from above.
To prove the rest, first suppose that $\sigma,b $
are infinitely differentiable in both $t$ and
$x$ with each derivative being bounded. In that case the result follows directly from Theorem
2.9.10 of \cite{Kr_77}. In the general case
take a $\zeta\in C^{\infty}_{0}(\bR)$
with unit integral
and for $\varepsilon>0$ introduce
$\zeta_{\varepsilon}(t)=\varepsilon^{-1}\zeta
(t/\varepsilon)$, $(\sigma^{\varepsilon},b
^{\varepsilon})(t,x)=(\sigma,b)(t,x)*\zeta_{\varepsilon}(t)$, where the convolution is performed with respect to $t$. Denote
by $x^{\varepsilon}_{t},\eta^{\varepsilon}_{t}$
the corresponding processes and set
$$
u^{\varepsilon}(t,x,\eta):=Ef\big((x^{\varepsilon}_{T-t},\eta^{\varepsilon}_{T-t})(t,x,\eta)\big).
$$
Since the assertions of the lemma are true for
$u^{\varepsilon}$, its derivative in $x,\eta$
admit the stated estimates (independent of $\varepsilon$) and then equation \eqref{6.21.3}
provides uniform in $\varepsilon$ estimates
of $\partial_{t}
u^{\varepsilon}(t,x,\eta)$.
 By Theorem 2.8.1
of \cite{Kr_77} $x^{\varepsilon}_{t},\eta^{\varepsilon}_{t}
\to
x _{t},\eta _{t}$, as $\varepsilon\downarrow0$, in such a sense that
$u^{\varepsilon}(t,x,\eta):=Ef\big((x^{\varepsilon}_{T-t},\eta^{\varepsilon}_{T-t})(t,x,\eta)\big)\to u^{\varepsilon}(t,x,\eta)$
at any point in $(-\infty,T]\times \bR^{2d}$.
By the results in \cite{Kr_77} also the derivatives in $\eta,x$ of $u^{\varepsilon}(t,x,\eta)$ converge to the corresponding
derivatives of $u (t,x,\eta)$.
By adding to this that, as is well known
$\sigma^{\varepsilon},b^{\varepsilon}$ and
their derivatives in $x$ converge to 
$\sigma ,b $ and their corresponding derivatives for every $x$ and almost any $t$,
we find in $(0,T)\times \bR^{2d}$  (a.e.) that
$$
\lim_{\varepsilon\downarrow0}\partial_{t}
u^{\varepsilon}(t,x,\eta)
=\check \cL(t,x,\eta)u(t,x,\eta).
$$
This easily proves (iii) and the lemma. \qed

Next, we take a  
nonnegative function $f(x,\eta)$, which
is a {\em polynomial\/} with respect to $\eta$
with   coefficients that are in $C^{\infty}_{0}(\bR^{d})$.
Then for $T\in(0,\infty)$ and $t\leq T$ denote 
$$
u(t,x,\eta)=Ef(x_{T-t}(t,x),\eta_{T-t}(t,x,\eta)).
$$ 
According to Lemma \ref{lemma 6.21.10}
the function $u(t,x,\eta)$ satisfies
\eqref{6.21.3} and, since $\eta_{T-t}(t,x,\eta)$
is affine in $\eta$, $u(t,x,\eta)$ is a polynomial
in $\eta$.
 
Recall that
$$
 \widehat {D\sigma }_{p_{0},\rho}:=
\sup_{r\leq \rho}r\sup_{C\in \bC_{r}}
\dashnorm D\sigma\|_{L_{p_{0}}(C)}.
$$
\begin{theorem}
                   \label{theorem 6.21.1}
 
Let $n\in\{ 1,2,...\}, \lambda\geq 0$.
Then there
are constants $\widehat{D\sigma},\hat b\in(0,1]$,
depending only on $d,\delta,p_{0}$,  
$n $,  and the power of the polynomial $f(x,\eta)$,
such that if 
\begin{equation}   
                             \label{12.10,1}
 \widehat {D\sigma }_{p_{0},\rho_{0}}\leq 
e^{- \lambda\rho_{0}}\widehat{D\sigma},\quad\hat b_{p_{0} , \rho_{0} }
\leq e^{- \lambda\rho_{0}}\hat b,
\end{equation}
 then
\begin{equation}
                              \label{12.20,5}
\int_{\bR^{d}}\sup_{|\eta|\leq 1}|u(0,x,\eta)|^{2n}e^{-\lambda|x|}\,dx
\leq Ne^{ \alpha T}\int_{\bR^{d}}\sup_{|\eta|\leq 1}|f(x,\eta)|^{2n}e^{-\lambda|x|}\,dx,
\end{equation}
where 
$$
\alpha=N\rho_{0}^{-2}e^{ \lambda\rho_{0}}
$$
and the constants called $N$ depend  only on $d,\delta,p_{0}$,  
$n $ and the power of the polynomial $f(x,\eta)$.

\end{theorem}

By taking $\lambda=0$ and using the arbitrariness of $\rho_{0}$
we come to the following.
\begin{corollary}
                             \label{corollary 2.6.1}
 If $\|D\sigma\|_{\dot E_{p_{0},1}}\leq \widehat{D\sigma}$
 and $\| b\|_{\dot E_{p_{0},1}}\leq \hat b$, where
 $\widehat{D\sigma},\hat b$ are taken from Theorem \ref{theorem 6.21.1},
 then
 $$
 \int_{\bR^{d}}\sup_{|\eta|\leq 1}|u(0,x,\eta)|^{2n} \,dx
\leq Ne^{N T}\int_{\bR^{d}}\sup_{|\eta|\leq 1}|f(x,\eta)|^{2n} \,dx,
$$
where the constants called $N$ depend  only on $d,\delta,p_{0}$,  
$n $ and the power of the polynomial $f(x,\eta)$.

\end{corollary}

The proof of Theorem \ref{theorem 6.21.1} is rather long
and we present  it in a separate section.

\mysection{Proof of Theorem \protect\ref{theorem 6.21.1}}
                       \label{section 1.14.1}

We need the following  which is similar to
Lemma 5.8 of \cite{Kr_25_1}.

\begin{lemma}
                         \label{lemma 12.8.1}

Let an integer $n\geq1$ and suppose that for 
$i=1,...,n$ we are given $p_{i} >0$,
integers $k_{i}\geq1$, and   polynomials $A_{i}(\eta)$ of degree $k_{i}$
on $\bR^{d}$. Then there exists a constant
$N=N(d, n,p_{i},\kappa_{i})$ such that
\begin{equation}
                        \label{12.18.1}
|A_{1}|^{p_{1}}\cdot...\cdot|A_{n}|^{p_{n}}\leq 
N\int_{B_{1}} | A_{1}(\eta))|^{p_{1}}\cdot...\cdot | A_{n}(\eta)|^{p_{n}}
\,d\eta,
\end{equation}
where $|A_{i}|$ is the maximum of absolute
values of the coefficients of $A$, written
without similar terms.
\end{lemma}

Proof. As it is not hard to see, it suffices
to prove that for any polynomial $A(\eta)$
of degree $k$ with $|A|=1$ and any $\gamma>0$ there exists $\varepsilon>0$,
depending only on $d,k,\gamma$,
such that
$$
|B_{1}\cap\{|A(\eta)|\leq\varepsilon \}
\leq  \gamma |B_{1}|.
$$
We are going to treat $A(\eta)$ as a random
variable on the probability space $(B_{1},dx/|B_{1}|)$. Observe that the set $\frA$
of the $A(\eta)$'s is compact in $C(\bar B_{1})$,
and, since for any polynomial its any level set  has Lebesgue measure zero, the distribution functions $F_{A}$ of the $A(\eta)$'s
form a compact set $\frF$ in $C[0,1]$. It follows
that for given $\gamma$ we can find a finite $\gamma/2$-net $F_{A_{1}},...,F_{A_{m}}$ in $\frF$ and $\varepsilon>0$ such that $F_{A_{i}}(\varepsilon)\leq\gamma/2$
for any $i=1,...,m$, and then for any $A\in\frA$
we can find $F_{A_{i}}$ such that
$$
F_{A}(|A(\eta)|\leq\varepsilon)
\leq F_{A_{i}}(|A(\eta)|\leq\varepsilon)
+\gamma/2\leq\gamma.
$$
\qed
 
Now we start proving the theorem.

{\em Step 1\/}.  As in the proof of Theorem \ref{theorem 11.26,3}
we assume that $\widehat { D\sigma }_{p_{0},\rho_{0}}
\leq 1$, $\hat b_{p_{0} , \rho_{0} }
\leq 1$. Then take a $C\in\bC_{\rho_{0}}$ and 
a nonnegative $\zeta
\in C^{\infty}_{0}(C)$  with the integral
of its square equal to one,
   multiply  \eqref{6.21.3}
by $\zeta^{2} u^{2n-1}$ and integrate by parts
with respect to $(t,x)$ regarding $\eta$
as a parameter in $B_{1}$.

Then as in the proof of Theorem \ref{theorem 11.26,3} (cf. \eqref{11.28,4}) for $s\leq T$ we find
$$
\int_{\bR^{d}}\zeta^{2}(s,x)u^{2n}(s,x,\eta)\,dx+
\int_{[s,T]\times \bR^{d}} \zeta^{2} u  ^{2n-2}
|Du |^{2}\,dxdt
$$
$$
\leq N\int_{\bR^{d}}\zeta^{2}(T,\cdot)f ^{2n }   \,dx
+N \int_{[s,T]\times \bR^{d}}
\big(|\partial_{t}\zeta^{2}|+|D\zeta |^{2}\big) u^{2n}\,dxdt
$$
\begin{equation}
                               \label{12.7,2}
+N\int_{[s,T]\times \bR^{d}} \zeta^{2}(|D\sigma|+|b|)^{2}u^{2n} \,dxdt+\int_{[s,T]\times \bR^{d}}\zeta^{2}u^{2n-1}F\,dxdt,
\end{equation}
where
$$
F=\sigma^{ik}\sigma_{(\eta)}^{jk}(t,x)u_{x^{i}\eta^{j}} (t,x,\eta)
$$
$$
+(1/2)\sigma_{(\eta)}^{ik} \sigma_{(\eta)}^{jk}(t,x)u_{\eta^{i}\eta^{j}}(t,x,\eta)
 +b^{i}_{(\eta)}(t,x)u_{\eta^{i}} (t,x,\eta).
$$

Observe that ($a^{n}(a^{n-1}b)\leq\varepsilon^{-1}a^{2n}+\varepsilon a^{2n-2}b^{2}$)
for any $\varepsilon>0$
$$
\int_{[s,T]\times \bR^{d}}\zeta^{2}u^{2n-1}\sigma^{ik}\sigma_{(\eta)}^{jk} u_{x^{i}\eta^{j}} \,dxdt\leq
N \int_{[s,T]\times \bR^{d}} \zeta^{2} |D\sigma| ^{2}u^{2n} \,dxdt
$$
$$
+\varepsilon\int_{[s,T]\times \bR^{d}} \zeta^{2}u^{2n-2}|u_{x\eta}|^{2} \,dxdt,
$$
where and below we allow the constants $N$
to also depend on $\varepsilon$. Also
  $$
\int_{[s,T]\times \bR^{d}}\zeta^{2}u^{2n-1}
\sigma_{(\eta)}^{ik} \sigma_{(\eta)}^{jk} u_{\eta^{i}\eta^{j}} \,dxdt
$$
$$
\leq N\int_{[s,T]\times \bR^{d}}\zeta^{2}  |D\sigma|^{2}u^{2n-1}|u_{\eta\eta}|\,dxdt,
$$
 $$
\int_{[s,T]\times \bR^{d}}\zeta^{2}u^{2n-1}
b^{i}_{(\eta)} u_{\eta^{i}}\,dxdt
$$
$$
=-(2n-1)\int_{[s,T]\times \bR^{d}}\zeta^{2}
\big(b^{i} u^{n-1}u_{\eta^{i}}\big)\big(u^{n-1}u_{x^{j}}\eta^{j}\big)\,dxdt
$$
$$
-2\int_{[s,T]\times \bR^{d}}\zeta\zeta_{x^{j}}\eta^{j}\big(u^{n}
b^{i}\big)\big(u^{n-1}  u_{\eta^{i}}\big)\,dxdt
-\int_{[s,T]\times \bR^{d}}\zeta^{2}u^{2n-1}b^{i} u_{x^{j}\eta^{i}}\eta^{j}\,dxdt
$$
$$
\leq\varepsilon\int_{[s,T]\times \bR^{d}}\zeta^{2}
u^{2n-2}(|Du|^{2}+|u_{x\eta}|^{2})\,dxdt+N\int_{[s,T]\times \bR^{d}}\zeta^{2}|b|^{2}u^{2n-2}|u_{\eta}|^{2}\,dxdt
$$
$$
+N\int_{[s,T]\times \bR^{d}}\zeta^{2}|b|^{2}u^{2n } \,dxdt+N\int_{[s,T]\times \bR^{d}}|D\zeta|^{2}u^{2n-2}|u_{\eta}|^{2}\,dxdt.
$$
We substitute these estimate into \eqref{12.7,2} and get
$$
\int_{\bR^{d}}\zeta^{2}(s,x)u^{2n}(s,x,\eta)\,dx+
\int_{[s,T]\times \bR^{d}} \zeta^{2} u  ^{2n-2}
(|Du |^{2}-\varepsilon|u_{x\eta}|^{2})\,dxdt
$$
$$
\leq N\int_{\bR^{d}}\zeta^{2}(T,\cdot)f ^{2n }   \,dx
+N \int_{[s,T]\times \bR^{d}}
\big(|\partial_{t}\zeta^{2}|+|D\zeta |^{2}\big) u^{2n-2}\big(u^{2}+|u_{\eta}|^{2}\big)\,dxdt
$$
$$
+N\int_{[s,T]\times \bR^{d}} \zeta^{2}(|D\sigma|+|b|)^{2}u^{2n-2}\big(u^{2 }+u |u_{\eta\eta}|+|u_{\eta}|^{2}|\big) \,dxdt . 
$$
By integrating through this inequality
with respect to $\eta$,
using Lemma \ref{lemma 12.8.1} and choosing
$\varepsilon$ appropriately we finally find
$$
\int_{\bR^{d}\times B_{1}}\zeta^{2}(s,x)u^{2n}(s,x,\eta)\,dxd\eta+
\int_{[s,T]\times \bR^{d}\times B_{1}} \zeta^{2} u  ^{2n-2}
 |Du |^{2} \,dxdtd\eta
$$
$$
\leq N\int_{\bR^{d}\times B_{1}}\zeta^{2}(T,\cdot)f ^{2n }   \,dxd\eta
+N \int_{[s,T]\times \bR^{d}\times B_{1}}
\big(|\partial_{t}\zeta^{2}|+|D\zeta |^{2}\big) u^{2n } \,dxdtd\eta
$$
\begin{equation}
                           \label{12.7,4}
+N\int_{[s,T]\times \bR^{d}\times B_{1}} \zeta^{2}(|D\sigma|+|b|)^{2}u^{2n } \,dxdtd\eta . 
\end{equation}

{\em Step 2\/}. Here we are dealing with the last term in \eqref{12.7,4}. Introduce $w=u^{n}$ and observe that for the function
$\zeta w$ as in the proof of Theorem \ref{theorem 11.26,3} we have 
$$
\partial_{t}(\zeta w)+\Delta(\zeta w)
 +2\zeta\Big(\frac{1}{n}-1\Big)a^{ij}(D_{i}(u^{n/2}))D_{j}(u^{n/2})
+G=0,
$$
where
$$
G=-w\partial_{t}\zeta-\Delta(\zeta w) 
+(1/2)\zeta a^{ij}D_{ij}w  
+\zeta b^{i}D_{i}w 
$$
$$ 
+n\zeta u^{n-1}\Big(\sigma^{ik}\sigma^{jk}_{(\eta)}u_{x^{i}\eta^{j}}+(1/2)\sigma^{ik}_{(\eta)}\sigma^{jk}_{(\eta)}u_{\eta^{i}\eta^{j}}
 +b^{i}_{(\eta)}u_{\eta^{i}}\Big).
$$
Then again as in the proof of Theorem \ref{theorem 11.26,3},  defining $w$ for $t>T$ as zero (and using that $w\geq0$), we conclude that for $t<T$
$$
0\leq\zeta w(t,x,\eta)\leq P_{2,4}G
=h(t,x,\eta) 
$$
$$
+\big(J_{1}+J_{2}+(1/2)J_{3}+J_{4}+n[J_{5}+
(1/2)J_{6}+J_{7} ]\big)(t,x,\eta),
$$
where   
$$
h(t,x,\eta)=E\big[\zeta(T,x_{T-t}(t,x))f^{n}
\big(x_{T-t}(t,x),\eta \big)\big],
$$
$$
J_{1}=-P_{2,4}(w\partial_{t}\zeta),\quad
J_{2}=-P_{2,4}(\Delta(\zeta w))=-\big(
P_{2,4}( (\zeta w)_{x^{i}})\big)_{x^{i}},
$$
$$
J_{3}=P_{2,4}\big(\zeta a^{ij}w_{x^{i}x^{j}}  \big),\quad J_{4}= P_{2,4}\big(\zeta b^{i}D_{i}w\big)
$$
$$
J_{5}=P_{2,4}\big(\zeta u^{n-1}\sigma^{ik}\sigma^{jk}_{(\eta)}u_{x^{i}\eta^{j}}\big),
\quad
J_{6}=P_{2,4}\big(\zeta u^{n-1}\sigma^{ik}_{(\eta)}\sigma^{jk}_{(\eta)}u_{\eta^{i}\eta^{j}}\big),
$$
$$
J_{7}=P_{2,4}\big(\zeta u^{n-1}b^{i}_{(\eta)}u_{\eta^{i}}\big)=(1/n)
\eta^{k}\big(P_{2,4}\big(\zeta  b^{i} w_{\eta^{i}})\big)_{x^{k}}
$$
$$
-(1/n)P_{2,4}\big(\zeta_{(\eta)}  b^{i} w_{\eta^{i}})-(1/n)P_{2,4}\big(\zeta   b^{i} w_{\eta^{i}(\eta)}).
$$

First, by Lemma \ref{lemma 11.29,1}
$$
\int_{[s,T]\times \bR^{d}}I_{C}b^{2}h^{2}
\,dxdt\leq N\hat b^{2}_{p_{0},\rho_{0}}
 \int_{\bR^{d}}\zeta^{2}(T,x)f^{2n}(x,\eta)\,dx.
$$

Next,
since $P_{2,4}=NP_{1,8}P_{1,8}$
and  $I_{C}\in\dot E_{p_{0}, 1}$ 
by Remark \ref{remark 11.30,1},  
by Theorem \ref{theorem 5.25,1} and
Corollary \ref{corollary 10.5,1}
(this combination will be used repeatedly
below)
$$
\int_{\bR^{d+1}_{s}}I_{C}|b|^{2}J_{1}^{2}\,dxdt
$$
\begin{equation}
                           \label{10.11,5}
\leq N\hat b^{2}_{p_{0},\rho_{0}}\int_{\bR^{d+1}_{s}} P^{2}_{1,8}(I_{C}w|\partial_{t}\zeta|)  \,dxdt
\leq N\hat b^{2}_{p_{0},\rho_{0}}\int_{\bR^{d+1}_{s}}  u^{2n}\rho_{0}^{2}|\partial_{t}\zeta|^{2}  \,dxdt,
\end{equation}

Then,
$
|J_{2}|
\leq NP_{1,8}(|D(\zeta w) |)
$,
 so that
\begin{equation}
                           \label{10.11,6}
   \int_{\bR^{d+1}_{s}}I_{C}|b|^{2}J_{2}^{2} \,dxdt
\leq N\hat b^{2}_{p_{0},\rho_{0}}\int_{\bR^{d+1}_{s}} |D(\zeta w) |^{2}\,dxdt.
\end{equation}

Dealing with $J_{3}$ observe that 
$$
P_{2,4}\big(\zeta  a^{ij} w_{x^{i}x^{j}}\big)=
\big[P_{2,4}\big(\zeta  a^{ij}w_{ x^{j}}\big)\big]_{x^{i}}
$$
$$
-P_{2,4}\big(\zeta_{x^{i}} a^{ij} w_{ x^{j}}\big)
-P_{2,4}\big(\zeta  [\sigma^{ik}_{x^{i}}\sigma^{jk}+\sigma^{ik}\sigma^{jk}_{x^{i}}] w_{ x^{j}}\big) .
$$
It follows that 
$$
 \int_{\bR^{d+1}_{s}}I_{C}|b|^{2}J_{3}^{2} \,dxdt\leq N \int_{\bR^{d+1}_{s}}I_{C}|b|^{2}P_{1,8}^{2}(\zeta|Dw |) \,dxdt
$$
$$
+N\hat b^{2}_{p_{0},\rho_{0}}\int_{\bR^{d+1}_{s}}
P_{1,8}^{2}\big(I_{C}|D\zeta |\,|Dw |
+I_{C}|D\sigma |\zeta|Dw |\big) \,dxdt
$$
\begin{equation}
                           \label{10.11,7}
 \leq N\hat b^{2}_{p_{0},\rho_{0}}\int_{\bR^{d+1}_{s}}
(\zeta^{2} +\rho_{0}^{2}|D\zeta|^{2} )|Dw|^{2}\,dxdt,
\end{equation}
where we used that $\widehat{D\sigma}_{p_{0},\rho_{0}}\leq 1$.

Next,
$$
\int_{\bR^{d+1}_{s}}I_{C}|b|^{2}J_{4}^{2} \,dxdt
$$ 
\begin{equation}
                           \label{10.13,2}
\leq N\hat b^{2}_{p_{0},\rho_{0}}
\int_{\bR^{d+1}_{s}}P_{1,8}^{2}(|b|\zeta|Dw|)\,dxdt\leq N \hat b^{2}_{p_{0},\rho_{0}}
\int_{\bR^{d+1}_{s}} \zeta^{2} u^{2n-2}|Du|^{2} \,dxdt,
\end{equation}

$$
 \int_{\bR^{d+1}_{s}}I_{C}|b|^{2}J_{5}^{2} \,dxdt\leq N\hat b^{2}_{p_{0},\rho_{0}}\int_{\bR^{d+1}_{s}}P^{2}_{1,8}\big(I_{C}|D\sigma|\zeta u^{n-1}  |u_{x\eta}|\big)\,dxdt
$$
\begin{equation}
                           \label{10.11,8}
\leq N\hat b^{2}_{p_{0},\rho_{0}}\int_{\bR^{d+1}_{s}}  \zeta^{2}   u^{2n-2}  |u_{x\eta}|^{2}\,dxdt.
\end{equation}
Similarly,
$$
 \int_{\bR^{d+1}_{s}}I_{C}|b|^{2}J_{6}^{2} \,dxdt\leq N\hat b^{2}_{p_{0},\rho_{0}}\int_{\bR^{d+1}_{s}}P^{2}_{1,8}\big(I_{C}|D\sigma|(\zeta|D\sigma| | u^{n-1}  |u_{\eta\eta}|)\big)\,dxdt
$$
\begin{equation}
                           \label{10.13,1}
\leq N\hat b^{2}_{p_{0},\rho_{0}}\int_{\bR^{d+1}_{s}} \zeta^{2} |D\sigma|^{2} | u^{2n -2}   |u_{\eta\eta}|^{2} \,dxdt.
\end{equation}

Finally,
$$
\int_{\bR^{d+1}_{s}}I_{C}|b|^{2}J_{7}^{2} \,dxdt\leq N\hat b^{2}_{p_{0},\rho_{0}}\int_{\bR^{d+1}_{s}} 
\zeta^{2}|b|^{2} |w_{\eta}|^{2}\,dxdt
$$
$$
+N\hat b^{2}_{p_{0},\rho_{0}} \int_{\bR^{d+1}_{s}} \big( 
|D\zeta|^{2}  |w_{\eta}|^{2}+\zeta^{2}|w_{\eta (\eta)}|^{2}\big)\,dxdt.
$$
  
Summing up the above estimates 
and using \eqref{1.21.3} yields 
$$
 \int_{\bR^{d+1}_{s}} \zeta^{2}|b|^{2}u^{2n}\,dxdt\leq N\hat b^{2}_{p_{0},\rho_{0}}
 \int_{\bR^{d}}\zeta^{2}(T,x)f^{2n}(x,\eta)\,dx
$$
$$
+N\hat b^{2}_{p_{0},\rho_{0}}
\int_{\bR^{d+1}_{s}}\Big(\rho_{0}^{-d-4}I_{C}
(u^{2n} +|w_{\eta}|^{2})
$$
$$
+\rho_{0}^{-d-2}I_{C}\big[u^{2n-2}\big(|Du|^{2} + |u_{x\eta}|^{2}\big)+|w_{\eta(\eta)}|^{2}\big]
$$
$$
+\zeta^{2}|D\sigma|^{2}u^{2n-2}|u_{\eta\eta}|^{2}
+\zeta^{2}  |b|^{2}|w_{\eta}|^{2}  \Big)\,dxdt.
$$

{\em Step 3\/}.
Obviously, similar  estimate is valid
for
$$
\int_{\bR^{d+1}_{s}}I_{C}|D\sigma|^{2}w^{2}\,dxdt,
$$
which after adding it to the above one,
integrating over $B_{1}$ with respect to $\eta$ and using Lemma \ref{lemma 12.8.1}
yields
 $$
 \int_{\bR^{d+1}_{s}\times B_{1}}\zeta^{2}(|b|^{2}+|D\sigma|^{2})u^{2n}\,dxdtd\eta
\leq N 
 \int_{\bR^{d}\times B_{1}}\zeta^{2}(T,\cdot)f^{2n} \,dxd\eta
$$
$$
+ N\Big(\hat b^{2}_{p_{0},\rho_{0}}+\widehat{D\sigma}_{p_{0},\rho_{0}}^{2}\Big)\rho_{0}^{-d-4}
\int_{\bR^{d+1}_{s}\times B_{1}}   
I_{C}u^{2n} 
\,dxdtd\eta
$$
$$
+N\Big(\hat b^{2}_{p_{0},\rho_{0}}+\widehat{D\sigma}_{p_{0},\rho_{0}}^{2}\Big)\rho_{0}^{-d-2}
\int_{\bR^{d+1}_{s}\times B_{1}}I_{C}|Du|^{2}\,dxdtd\eta
$$
$$
+N_{1}\Big(\hat b^{2}_{p_{0},\rho_{0}}+\widehat{D\sigma}_{p_{0},\rho_{0}}^{2}\Big)
\int_{\bR^{d+1}_{s}\times B_{1}}
\zeta^{2}(|b|^{2}+|D\sigma|^{2})u^{2n}\,dxdtd\eta.
$$

For
\begin{equation}
                                 \label{12.9,6}
N_{1}\Big(\hat b^{2}_{p_{0},\rho_{0}}+\widehat{D\sigma}_{p_{0},\rho_{0}}^{2}\Big)
\leq 1/2
\end{equation}
this implies that
 $$
 \int_{\bR^{d+1}_{s}\times B_{1}}\zeta^{2}(|b|^{2}+|D\sigma|^{2})u^{2n}\,dxdtd\eta
\leq N 
 \int_{\bR^{d}\times B_{1}}\zeta^{2}(T,\cdot)f^{2n} \,dxd\eta
$$
$$
+ N \rho_{0}^{-d-4}
\int_{\bR^{d+1}_{s}\times B_{1}}   I_{C}
 u^{2n} 
\,dxdtd\eta
$$
$$
+N\Big(\hat b^{2}_{p_{0},\rho_{0}}+\widehat{D\sigma}_{p_{0},\rho_{0}}^{2}\Big)\rho_{0}^{-d-2}
\int_{\bR^{d+1}_{s}\times B_{1}}I_{C} u^{2n-2}|Du|^{2}\,dxdtd\eta.
$$
Coming back to \eqref{12.7,4} we get
$$
\int_{\bR^{d}\times B_{1}}\zeta^{2}(s,x)u^{2n}(s,x,\eta)\,dxd\eta+
\int_{[s,T]\times \bR^{d}\times B_{1}} \zeta^{2} u  ^{2n-2}
 |Du |^{2} \,dxdtd\eta
$$
$$
\leq N\int_{\bR^{d}\times B_{1}}\zeta^{2}(T,\cdot)f ^{2n }   \,dxd\eta
+N\rho_{0}^{-d-4} \int_{[s,T]\times \bR^{d}\times B_{1}}I_{C} u^{2n } \,dxdtd\eta
$$
$$
+N\Big(\hat b^{2}_{p_{0},\rho_{0}}+\widehat{D\sigma}_{p_{0},\rho_{0}}^{2}\Big)\rho_{0}^{-d-2}
\int_{\bR^{d+1}_{s}\times B_{1}} I_{C}u^{2n-2}|Du|^{2}\,dxdtd\eta.
$$

After that we repeat the same manipulations
as at the end of the proof of Theorem \ref{theorem 11.26,3}
and similarly to \eqref{12.9,4} find
$$
e^{\lambda\rho_{0}}\int_{\bR^{d}\times B_{1}} u^{2n}(s,x,\eta)e^{-\lambda|x|}\,dxd\eta
+
\int_{[s,T]\times \bR^{d}\times B_{1}}e^{-\lambda|x|} u  ^{2n-2}
 |Du |^{2} \,dxdtd\eta
$$
$$
\leq Ne^{ \lambda\rho_{0}}\int_{\bR^{d}\times B_{1}}e^{-\lambda|x|}f ^{2n }   \,dxd\eta
$$
$$
+Ne^{2\lambda\rho_{0}}\rho_{0}^{-2} \int_{[s,T]\times \bR^{d}\times B_{1}}e^{-\lambda|x|} u^{2n } \,dxdtd\eta
$$
$$
+N_{2}e^{2\lambda\rho_{0}}\Big(\hat b^{2}_{p_{0},\rho_{0}}+\widehat{D\sigma}_{p_{0},\rho_{0}}^{2}\Big)
\int_{\bR^{d+1}_{s}\times B_{1}} e^{-\lambda|x|}u^{2n-2}|Du|^{2}\,dxdtd\eta.
$$

Now along with \eqref{12.9,6} we require
$$
 N_{2}e^{2\lambda\rho_{0}}\Big(\hat b^{2}_{p_{0},\rho_{0}}+\widehat{D\sigma}_{p_{0},\rho_{0}}^{2}\Big)
\leq 1.
$$
Then
$$
\int_{\bR^{d}\times B_{1}} u^{2n}(s,x,\eta)e^{-\lambda|x|}\,dxd\eta
\leq N \int_{\bR^{d}\times B_{1}}e^{-\lambda|x|}f ^{2n }   \,dxd\eta
$$
$$
+Ne^{ \lambda\rho_{0}}\rho_{0}^{-2} \int_{[s,T]\times \bR^{d}
\times B_{1}}e^{-\lambda|x|} u^{2n } \,dxdtd\eta,
$$ 
and \eqref{12.20,5} follows.
\qed   
 
\mysection{Proof of Theorem \ref{theorem 3.7.1}}

First,
we need a version of one of Kolmogorov's
results. Actually, in this paper
where everything is nice and smooth we
do not need this version in full generality.
It will be indeed needed when in the future
we will approximate not so regular coefficients
and pass to the limit in the situation
when the sequence of random fields will converge
only at each point in probability.

Let $n,r,c_{1},...,c_{r}\geq 2$ be integers and set
$c=(c_{1},...,c_{r})$.
Denote by $\bZ^{r} _{n}(c)$
 the lattice
in $[0,1]^{r} $ consisting of  points
$z=( z^{1}c_{1}^{-n},...,z^{r}c_{r}^{-n}) $, where $z^{i}=0,1,2,...,c_{i} ^{n} $. Define
$$
\bZ^{r} _{\infty}(c)=\bigcup_{n}\bZ^{r} _{n}(c)
$$

\begin{lemma}
                                               \label{lemma 3.5.1} 
Let  $u(z)$ be  a real-valued
function defined for $z\in\bZ^{r} _{\infty}(c)$. Assume that there exist a number
$\alpha >0$ and an integer $n\geq0$ such that for any $m\geq n$, $z_{1},z_{2}
\in \bZ^{r}_{m}(c)$, $i=1,...,r$, such that
$ |z^{i}_{1}-z^{i}_{2}|\leq c_{i}^{-m}$,
$z^{j}_{1}=z^{j}_{2}$ for $j\ne i$, we have $|u(z_{1})-u(z_{2})|\leq 2^{-m \alpha}$. Then for any $z_{1},z_{2}\in \bZ^{r} _{\infty}(c)$,
satisfying  
\begin{equation}
                      \label{3.5.5}
   |z^{i}_{1}-z^{i}_{2}|\leq c_{i}^{-n-1},\quad i=1,...,r,
\end{equation}
 we have  
\begin{equation}
                     \label{3.2.1}
 |u(z_{1})-u(z_{2}) |\leq
N\max_{i}|z^{i}_{1}-z^{i}_{2}|^{\alpha_{i}},
\quad \alpha_{i}=\alpha\ln 2/\ln c_{i},
\end{equation}
 where $N$ depends only on $r,c_{1},...,c_{r}$, and $\alpha$.
Furthermore, by continuity $u$
extends on $[0,1]^{r}$ and,
for the extension, called again $u$,
\eqref{3.2.1} holds for any 
$z_{1},z_{2}\in [0,1]^{r}$ satisfying
\eqref{3.5.5}.
  
\end{lemma}

Proof. The second assertion is obvious and to prove the first one let  $z_{1},z_{2}\in \bZ^{r} _{\infty}(c)$, $z_{1}=(z_{1}^{1},...,z_{1}^{r})$,
$z_{2}=(z_{2}^{1},...,z_{2}^{r})$. Represent $z_{1}^{i}$ and $z_{2}^{i}$ 
as  
 $$
z_{1}^{i}=\sum_{j=0}^{\infty}\varepsilon^{ij}_{1}c_{i}^{-j},\quad
z_{2}^{i}=\sum_{j=0}^{\infty}\varepsilon^{ij}_{2}c_{i}^{-j},\quad
\varepsilon^{ij}_{k}\in\{0,1,...,c_{i}-1\},
 $$
take some integer  $m  
\geq n$,  
 define $z^{i}_{1m}$ and $z^{i}_{2m}$ as the above sums for $j\leq m $, and let
 $$
z_{1m}=(z^{1}_{1m},...,z^{r}_{1m}),\quad z_{2m}=(z^{1}_{2m},
...,z^{r}_{2m}).
 $$
Then, as it is not hard to see, 
\begin{equation}
                     \label{3.5.3}
  |z^{i}_{1}-z^{i}_{2}|\leq c_{i}^{-m },
\quad i=1,...,r,
\end{equation}
 implies that   $|z_{1m}^{i}  -z^{i}_{2m}|
\leq 2c_{i}^{-m}$.  Since
  $z_{1m},z_{2m}\in \bZ^{r}_{m}(c)$,
\begin{equation}
                       \label{3.5.1}
|u(z_{1m})-u(z_{2m}) |\leq
2r2^{-ma}.
\end{equation}

Furthermore, 
$$
|u(z_{1})-u(z_{1 m })|\leq
\sum_{k=m}^{\infty}|u(z_{1,(k+1)})-u(z_{1k})|.
$$
Here $z _{1,(k+1)},z _{1k}\in \bZ^{r}_{k+1}(c)$ and
 $|z^{i}_{1,(k+1)}-z^{i}_{1k}|\leq
(c_{i}-1)c_{i}^{-(k+1)}$ for any $i$. It follows that 
$$
|u(z_{1,(k+1)})-u(z_{1k})|
\leq r\max c_{i}2^{-(k+1)\alpha},
$$
$$
|u(z_{1})-u(z_{1m})|\leq N
 2^{-m\alpha }.
$$
Combined with the similar estimate
for $z_{2}$ and \eqref{3.5.1} this
leads to
\begin{equation}
                       \label{3.5.2}
|u(z_{1 })-u(z_{2 }) |\leq N
 2^{-m\alpha }.
\end{equation}
 
We proved this estimate assuming
\eqref{3.5.3} and that $m\geq n$. Now take $z_{1},z_{2}\in \bigcup_{m}\bZ^{r} _{m}$
such that \eqref{3.5.5} holds and define
$$
m=\min_{i}\lfloor \ln (1/|z^{i}_{1}-z^{i}_{2}|) /
\ln c_{i}\rfloor.
$$
Then it is not hard to see that $m\geq n$ and \eqref{3.5.3} holds.

After that it only remains to observe that
$$
2^{-m\alpha}\leq 2\max_{i}2^{\alpha\ln |z^{i}_{1}-z^{i}_{2}|/\ln c_{i}}
=2\max_{i}|z^{i}_{1}-z^{i}_{2}|^{\alpha\ln 2/\ln c_{i}}.
$$
The lemma is proved.  \qed

In Lemma \ref{lemma 3.5.1} 
set $r=d+1$, $c_{1}=4,c_{2}=...=c_{d+1}=2$ and, accordingly, introduce
$\bZ^{d+1}_{n}(c)$, $\bZ^{d+1}_{\infty}(c)$. Also set $\bZ^{d}_{m}(2) =\bZ^{d}_{m}(2,...,2)$ with 2 repeated $d$-times. Recall that $\kappa>(d+2)/2$.
\begin{lemma}
                                  \label{lemma 3.5.2}
Let a   random field $u(t,x)$ be defined on $\bZ^{d+1}_{\infty}(c)$.
Assume that there exist constants $ \gamma\geq2\kappa $, $K<\infty$
such that  for $t,s\in
\bZ^{1}_{\infty}(4),x\in\bZ^{d}_{\infty}(2)$
$$
E|u(t,x)-u(s,x)|^{\gamma}\leq 
K^{\gamma} |t-s |^{\gamma/2}, 
$$
$$
E\sup_{x,y\in \bZ^{d}_{\infty}(2)}\frac{|u(t,x)-u(t,y)|^{2\kappa}}
{|x-y|^{2\kappa-d}}\leq K^{2\kappa}.
$$
 Then, for every $0<\alpha<1-(d+2)/(2\kappa)$  
with probability one there exists
a continuous extension of $u$ on
$[0,1]^{d+1}$, called again $u$, and
an integer-valued $n=n(\omega,\alpha,\gamma,\kappa,d)$
such that for any $(t,x),(s,y)\in
[0,1]^{d+1}$ satisfying $|t-s|\leq 2^{-n}$ and $\|x-y\|\leq 2^{-n}$, where
$\|x-y\|=\max_{i}|x^{i}-y^{i}|$, we have
\begin{equation}
                   \label{3.3.5}
|u(t,x)-u(s,y)|\leq N(\alpha, d)K
(|t-s|^{\alpha/2}+|x-y|^{\alpha}).
\end{equation}
 
\end{lemma}

Proof. Obviously, we may assume that $K=1$ and we only need to concentrate
on $(t,x),(s,y)\in
\bZ^{d+1}_{\infty}(c)$. Then  for   any integer $n\geq0$ 
introduce the event $A_{n} $
consisting of $\omega$
such that   

(i)
for any $m\geq n$, $(t,x),(s,x)
\in \bZ^{d+1}_{m}(c)$,   for which
$ |t-s|\leq 4^{-m}$,
 we have $|u(t,x)-u(s,x)|\leq 2^{-m \alpha}$, and

 (ii) for any $m\geq n$, $(t,x),(t,y)
\in \bZ^{d+1}_{m}(c)$,   for which
$ \|x-y\|\leq 2^{-m}$,
 we have $|u(t,x)-u(t,y)|\leq 2^{-m \alpha}$.

According to Lemma \ref{lemma 3.5.1} 
on $A_{n} $ we have
$$
 |u(t,x)-u(s,y) |\leq
N(|t-s|^{\alpha/2}+|x-y|^{\alpha}),
$$
where $N$ depends only on $d $, and $\alpha$,
as long as $(t,x),(s,y)\in
[0,1]^{d+1}$ and $|t-s|+\|x-y\|\leq 2^{-n-1}$. Therefore, to prove the lemma
it suffices to show that
\begin{equation}
                          \label{3.6.1}
\sum_{n}P(A_{n}^{c})<\infty.
\end{equation}

The complement $A^{c}_{n} $ of $A_{n} $
is the union of events such that either

(i) there are $m\geq n$,
 $x\in \bZ^{d}_{m}(2)$ $(=\bZ^{d}(2,...,2)$) and    $t,s\in \bZ^{1}_{m}(4)$ for which
$ |t-s|\leq 4^{-m}$ 
and $|u(t,x)-u(s,x)|^{\beta}\geq 2^{-m\alpha \beta}$, or

(ii) there are $m\geq n$, $t\in \bZ^{1}_{m}(c_{1})$ and   $x,y\in \bZ^{d}_{m}(2)$ for which
$ \|x-y\|\leq 2^{-m}$ and
 $|u(t,x)-u(t,y)|^{2\kappa}\geq 2^{-m 2\kappa \alpha}$.

   For
each $m\geq n$ the number of events in group (i) is not greater than the
number of points in $\bZ^{d}_{m}(2)$
times the number of couples
$t,s\in \bZ^{1}_{m}(4)$ for which
$ |t-s|\leq 4^{-m}$, so the total number is less than $2(4^{m}+1)(2^{m}+1)^{d}\leq N4^{m}2^{md}$. By Chebyshev's inequality 
the probability
of each event in group (i)
with fixed $m\geq n$ is less than
$N2^{m\alpha \gamma}4^{-m \gamma/2}$.
Therefore, the probability of event (i) is
less than
$$
N\sum_{m\geq n}2^{m(d+2+\alpha \gamma-\gamma)} =N2^{n(d+2+\alpha \gamma-\gamma)},
$$
where the equality follows from the fact that $d+2+\alpha \gamma-\gamma
<d+2-\gamma (d+2)/2\kappa)\leq0$.
The sum over $n$ of these estimates is finite.

For
each $m\geq n$ and $t\in \bZ^{1}_{m}( 4)$ the event that there
are $x,y\in \bZ^{d}_{m}(2)$ for which
$ \|x-y\|\leq 2^{-m}$ and
 $|u(t,x)-u(t,y)|^{2\kappa}\geq 2^{-m 2\kappa \alpha}$ is a subset of the event that
$$
 \sup_{x,y\in [0,1]^{d}}\frac{|u(t,x)-u(t,y)|^{2\kappa}}
{|x-y|^{2\kappa-d}}\geq c
2^{-m 2\kappa \alpha +m(2\kappa-d)}
=c2^{m2\kappa(1-\alpha)-md},
$$
where $c=c(d,\kappa)>0$.
Therefore, for each $m\geq n$ the probability of the part of event (ii)
corresponding to $m$
is less than
$N2^{2m-m2\kappa(1-\alpha)+md}$,
where $2+d+2\kappa(\alpha-1)<0$.
This implies that the probability of event (ii) is less than $N2^{2n-n2\kappa(1-\alpha)+nd}$ and along with
the above analysis of event (i) yields
\eqref{3.6.1}. The lemma is proved.

 {\bf Proof of Theorem \ref{theorem 3.7.1}}. As  it is known since \cite{BF_61}   (1961)
(see also \cite{Ku_90} 1990), owing to the smoothness of $\sigma,b$, 
one can define $x_{s}(t,x)$ in such a way
that it becomes differentiable in $x$ for all $(\omega,t,s)$
and the derivative $\eta^{i}D_{i}x_{s}(t,x)$ of $x_{s}(t,x)$ 
in the direction of $\eta$
satisfies the same equation as $\xi_{s}(t,x,\eta)$ from Lemma \ref{lemma 6.21.1},
for which \eqref{6.21.5} holds.
Hence, for any $(t,x)$ with probability one 
$\xi_{s}(t,x,\eta)=\eta^{i}D_{i}x_{s}(t,x)$
for all $s\geq0$.

It follows that if $f(x)$ is smooth with compact support and 
$$
v(t,x,\eta,s):=  
E\Big|\big((Df)(x_{s}(t,x)),\xi_{s}(t,x,\eta)\big)\Big|^{2\kappa}
$$
 ($(\cdot,\cdot)$ is the scalar product in $\bR^{d}$), then   
by Theorem \ref{theorem 6.21.1}, with $n=1$ there, 
 for $|t|,s\leq T$ we have
$$
\int_{\bR^{ d}}e^{-|x|} \sup_{|\eta|\leq1}v^{2}(t,x,\eta,s) \,dx  
$$
\begin{equation}
                                                        \label{7.2.4}
\leq 
N\int_{\bR^{ d}}e^{-|x|} \sup_{|\eta|\leq1} |f_{(\eta)}(x)| ^{4\kappa }\,dx 
\leq N  \int_{\bR^{ d}}e^{-|x|}|Df (x)|^{4\kappa }\,dx ,
\end{equation}
where (and below) the  constants $N$, 
depend  only on  $d$, $\delta$, $p_{0}, \rho_{0}$,  $\kappa$, $T$.
Next,  
$$
 E\int_{\bR^{d}}e^{-|x|}|D\big(f(x_{s}(t,x))\big)|^{2\kappa }\,dx 
\leq N E\int_{\bR^{d}\times B_{1}}e^{-|x|}| \big(f(x_{s}(t,x))\big)_{(\eta)}|^{2\kappa }\,dxd\eta
$$
$$
= N\int_{\bR^{d}\times B_{1}}e^{-|x|}
v (t,x,\eta,s) \,dx  d\eta.
$$

By using \eqref{7.2.4} and H\"older's inequality we obtain that
$$
 E\int_{\bR^{d}}e^{-|x|}|D\big(f(x_{t}(x))\big)|^{2\kappa }\,dx \leq
N  \Big(\int_{\bR^{ d}}e^{-|x|}|Df (x)|^{4\kappa }\,dx\Big)^{1/2} .
$$

We obtained this estimate for smooth $f$
with compact support. By using Fatou's lemma it is extended to all smooth functions. Clearly, one can also take
$\bR^{d}$-valued $f$'s. For $f(x)\equiv x$ we get
\begin{equation}
                       \label{2.8.5}
 E\int_{\bR^{d}}e^{-|x|}|D (x_{s}(t,x) \big)|^{2\kappa }\,dx \leq
N  \Big(\int_{\bR^{ d}}e^{-|x|} \,dx\Big)^{1/2}=:M,
\end{equation}
which proves \eqref{3.7.6}.

By Morrey's theorem (see, for instance, Theorem 
10.2.1 of \cite{Kr_08}) this implies that ($\kappa >d/2$)
\begin{equation}
                                                        \label{3.7.2}
E \sup_{x,y\in [0,1]^{d}}\frac{| x_{s}(t,x)  - x_{s}(t,y)  |^{2\kappa }
}{|x-y|^{2\kappa -d}} \leq
N  M  .
\end{equation} 
Furthermore, owing to Theorem 3.17 of \cite{Kr_24}
for any $\gamma>0$, $s_{1},s_{2}\leq 1$, 
\begin{equation}
                      \label{3.7.1}
E| x_{s_{1}}(t,x)  - x_{s_{2}}(t,x)  |^{\gamma}
\leq  N(d,\delta,\gamma,p_{0},\rho_{0})
 |s_{1}-s_{2} |^{\gamma/2}.
\end{equation}

Now the assertions of the theorem about
the H\"older continuity of $x_{s}(t,x)$ in $(s,x)$
for $(s,x)\in[0,1]^{d+1}$, $|t|\leq T$,
follow from Lemma \ref{lemma 3.5.2}.
Clearly, it extends over more general sets.
The theorem is proved.

\begin{remark}
                        \label{remark 1.31.6}
  Similarly to Remark \ref{remark 1.31.5}
the result of Theorem \ref{theorem 3.7.1}
is valid with obvious changes if
$$
\int_{\bR}\sup_{\bR^{d}}(|D\sigma|+|b|)^{2}(t,x)\,
dt<\infty.
$$
Indeed, in that case one can short cut the proof of Theorem \ref{theorem 6.21.1}
after \eqref{12.7,4}.

\end{remark}

{\bf Acknowledgments}. The author is thankful
to D. Kinzebulatov for several discussions and  
resulting corrections.

{\bf Declarations}. 
No funds, grants, or other support was received.
The author has no relevant financial or non-financial interests to disclose.
The manuscript contains no data.

\end{document}